\documentclass[11pt,a4paper]{amsart}
\usepackage{amscd,amssymb,amsopn,amsmath,amsthm,graphics,amsfonts,enumerate,verbatim,calc}
\usepackage[dvips]{graphicx}
\usepackage{mathrsfs} 
\usepackage{color}

\usepackage{url}

\usepackage{amssymb,amsmath}

\headsep=1cm \numberwithin{equation}{section}
\newtheorem{theorem}{Theorem}[section]

\newtheorem{proposition}[theorem]{Proposition}

\newtheorem{question}[theorem]{Question}

\theoremstyle{definition}
\newtheorem{definition}[theorem]{Definition}

\theoremstyle{remark}
\newtheorem{remark}[theorem]{Remark}

\newtheorem{acknowledgement}{Acknowledgement}

\newcommand{\fm}{\frak{m}}
\newcommand{\fp}{\frak{p}}

\newcommand{\fM}{\frak{M}}

\begin{document}
\title[Matijevic-Roberts type theorems, Rees rings and associated graded rings]
{Matijevic-Roberts type theorems, Rees rings and associated graded rings}

\author[J. Horiuchi]{Jun Horiuchi}
\address{Department of Mathematics, Nippon Institute of Technology, Miyashiro, Saitama 345-8501, Japan}
\email{jhoriuchi.math@gmail.com}

\author[K. Shimomoto]{Kazuma Shimomoto}
\address{Department of Mathematics, Institute of Science Tokyo, 2-12-1 Ookayama, Meguro, Tokyo 152-8551, Japan}
\email{shimomotokazuma@gmail.com}

\thanks{2020 {\em Mathematics Subject Classification\/}: 13A02, 13A30, 13F45, 13E05}

\keywords{Associated graded ring, integral closure, Rees ring, seminormality, weak normality}

%\subjclass{13}
%\subjclass[2000]{Primary 13-XX}
%\subjclass[2000]{Primary ; Secondary}
%\date{\today \, (\printtime)}
%\date{\today}

\begin{abstract}
The aim of this article is to investigate interrelated structures lying among three notable problems in commutative algebra. These are Lifting problem, Ascent/descent along associated graded rings, and Matijevic-Roberts type problem.
\end{abstract}

\maketitle 

\begin{center}
{\textit{In memory of Professor Shiro Goto}}
\end{center}

\section{Introduction}

Historically, Rees rings and associated graded rings of Noetherian rings were studied in connection with the problem of resolving singularities of Noetherian schemes by blowing up closed subschemes. The aim of the present article is to study these objects from the ring-theoretic viewpoint. One often encounters a situation where many notable ring-theoretic properties can be derived by studying associated graded rings. Geometrically, associated graded rings correspond to exceptional divisors of a blown-up scheme. Let $A$ be a Noetherian ring with an ideal $I \subseteq A$ and let $\mathscr R_+(I)$ be its Rees ring. For example, Barshay proved that if $A$ is Cohen-Macaulay and $I$ is generated by a regular sequence, then $\mathscr R_+(I)$ is also Cohen-Macaulay in \cite{Bar73}. Let $G(I)$ be the associated graded ring of $A$ with respect to $I$. Then Goto and Shimoda characterized the Cohen-Macaulayness/Gorensteinness of $\mathscr R_+(I)$ in terms of the corresponding property of $G(I)$ and its attached $a$-invariant in \cite{GS79}. Conversely, one can ask how the singularities of $G(I)$ affect the singularities of $A$. Suppose that $\mathcal{P}$ is one of the following properties: Cohen-Macaulay, Gorenstein, reduced, domain, or normal. Then under a very mild condition, if $G(I)$ has $\mathcal{P}$, then so does $A$ (see \cite[Theorem 4.5.7, Theorem 4.5.8 and Theorem 4.5.9]{BrHer93} for the proofs). We use Rees rings, extended Rees rings, and associated graded rings to study the relationship among three notable problems in commutative algebra, which are Lifting Problem, Ascent/descent Problem in associated graded rings, and Matijevic-Roberts type Problem. Our guiding principle is explained by Proposition \ref{local-lifting} and Proposition \ref{local-lifting2}. As a test example, we consider the case $\mathcal{P}=\rm{weakly~normal,~seminormal}$. See Theorem \ref{Matijevic-Roberts1} and Theorem \ref{grlocal1}.

\section{Notation}

Let us fix some notation. Let $A$ be a commutative Noetherian ring with its proper ideal $I \subseteq A$. Let $t$ be an indeterminate over $A$. For any integers $n \le 0$, we set $I^n =A$ by convention. The \textit{Rees ring of $I$} is defined as $
\mathscr R_+(I):=A[It]=\bigoplus_{n\geq 0}I^n t^n\subseteq A[t]$, \textit{extended Rees ring of $I$} as $\mathscr R(I):=A[It, t^{-1}]=\bigoplus_{n\in \mathbb{Z}}I^n t^n\subseteq A[t, t^{-1}]$, and the \textit{associated graded ring of $I$} as $G(I):=\mathscr R_+(I)/I\mathscr R_+(I)=\mathscr R(I)/t^{-1} \mathscr R(I)=\bigoplus_{n\geq 0}I^n/I^{n+1}$. For an ideal $I \subseteq A$, we say that $I$ is \textit{normal} if $I^n$ is integrally closed for all $n>0$. For a graded ring $A=\bigoplus_{n \in \mathbb{Z}}A_n$ with an ideal $I \subseteq A$, let $I^*$ denote the ideal of $A$ that is generated by all homogeneous elements contained in $I$.

As we already mentioned in the introduction, a remarkable aspect of associated graded rings is that one can often derive a certain property of $A$ from $G(I)$. For instance, it is known that if the associated graded ring with respect to some proper ideal of $A$ is Cohen-Macaulay, then $A$ is also Cohen-Macaulay. A similar statement holds for the Gorenstein property (see \cite[Proposition 1.2]{O93}). Let us recall some definitions. See a survey paper \cite{V11} for other equivalent definitions of seminormality and weak normality.

\begin{definition}
Let $A$ be a commutative reduced Noetherian ring. Consider the following conditions.
\begin{enumerate}
\item[\rm{(i)}]
For any elements $y, z\in A$ with $y^3=z^2$, there is an element $x \in A$ satisfying $y=x^2, z=x^3$.

\item[\rm{(ii)}]
For any elements $y,z,w \in A$ and any nonzero divisor $d \in A$ with $z^p=yd^p$ and $pz=dw$ for some prime integer $p$, there is an element $x \in A$ with $y=x^p$ and $w=px$.
\end{enumerate}

A ring $A$ which satisfies the condition $\rm{(i)}$ is called \textit{seminormal} (see \cite[Definition 2.17.]{V11}), and which satisfies both conditions $\rm{(i)}$ and $\rm{(ii)}$ is called \textit{weakly normal} (see \cite[Definition 3.12.]{V11}).
\end{definition}

\section{Lifting problem and associated graded rings}

$\mathbf{Set~up}$: Let $\mathbf{C}$ be the category of Noetherian rings. We consider a subcategory $\mathbf{D}$ of $\mathbf{C}$ satisfying the following condition.
\begin{enumerate}
\item[$\bullet$]
Let $A \in \mathbf{D}$. Then any finitely generated $A$-algebra belongs $\mathbf{D}$, the localization of $A$ with respect to any multiplicative subset of $A$ belongs to $\mathbf{D}$, and any subring $B \subseteq A$ such that $A$ is faithfully flat over $B$ belongs to $\mathbf{D}$. Let $I \subseteq A$ be an ideal. Then the (extended) Rees ring of $A$ with respect to $I$ belongs to $\mathbf{D}$.
\end{enumerate}

We note that if $A \to B$ is faithfully flat and $B$ is Noetherian, then $A$ is also Noetherian. Henceforth, we fix $\mathbf{D}$ as above and consider a ring-theoretic property $\mathcal{P}$ which is defined on any object of $\mathbf{D}$. We consider the following properties on the objects from $\mathbf{D}$.

\begin{enumerate}
\item[($\bf{P1}$)]
$A$ has $\mathcal{P}$ if and only if the same holds on $A_\fp$ for every prime ideal $\fp \subseteq A$.

\item[($\bf{P2}$)]
If $A$ has $\mathcal{P}$, then the polynomial algebra $A[X]$ has $\mathcal{P}$.

\item[($\bf{P3}$)]
If $A$ has $\mathcal{P}$ and $B \to A$ is a faithfully flat extension of rings, then $B$ has $\mathcal{P}$.
\end{enumerate}

In addition to the above properties, we consider the following problems.

\begin{enumerate}
\item[($\bf{Lift}$)]
Let $(A,\fm)$ be a local ring. If $A/yA$ has $\mathcal{P}$ for any nonzero divisor $y \in \fm$, then so does $A$ (called ``lifting property" for $\mathcal{P}$~).

\item[($\bf{Gr}$)]
Let $(A,\fm)$ be a local ring with an ideal $I \subseteq \fm$. If $G(I)=\bigoplus_{n \ge 0}I^n/I^{n+1}$ has $\mathcal{P}$, then so does $A$ (called ``ascent property along associated graded rings" for $\mathcal{P}$~).

\item[($\bf{MR}$)]
Let $A=\bigoplus_{n \in \mathbb{Z}}A_n$ be a graded Noetherian ring. Then $A$ has $\mathcal{P}$ if and only if the localization $A_\fm$ has $\mathcal{P}$ for every graded maximal ideal $\fm \subseteq A$ (called ``Matijevic-Roberts type theorem" for $\mathcal{P}$).
\end{enumerate}

We prove two fundamental results (Proposition \ref{local-lifting} and Proposition \ref{local-lifting2} below) which formulate how Problem $(\bf{Lift})$, Problem $(\bf{Gr})$ and Problem $(\bf{MR})$ are related to one another.

\begin{proposition}
\label{local-lifting}
Fix a subcategory $\mathbf{D} \subseteq \mathbf{C}$. Suppose that $\mathcal{P}$ satisfies $(\bf{P2})$ in $\mathbf{D}$ and that $(\bf{Gr})$ is solved for an ideal $I=yA$ for a nonzero divisor $y \in \fm$ for a local ring $(A,\fm)$ in $\mathbf{D}$. Then $(\bf{Lift})$ is solved in $\mathbf{D}$.
\end{proposition}

\begin{proof}
Let $(A,\fm)$ be a local ring in $\mathbf{D}$ and let $y\in \fm$ be a nonzero divisor such that $A/yA$ has $\mathcal{P}$. Then we have a graded ring map:
$$
\phi:(A/yA)[X] \to G(yA)=\bigoplus_{n \ge 0} y^nA/y^{n+1}A
$$
by letting $X \mapsto \overline{y} \in yA/y^2A$. Then $\phi$ is an isomorphism in view of \cite[Theorem 16.2]{M86}. Since $A/yA$ has $\mathcal{P}$ by assumption,  the condition $(\bf{P2})$ allows us to say that $(A/yA)[X]$ has $\mathcal{P}$. By the graded isomorphism $\phi$, we see that the associated graded ring $G(yA)$ has $\mathcal{P}$. By virtue of assumption $(\bf{Gr})$, it follows that $A$ has $\mathcal{P}$. In other words, $(\bf{Lift})$ is solved.
\end{proof}

We recall the following result (see \cite[Proposition I.7.4]{Sey72}).

\begin{theorem}[Seydi]
\label{NormalRing}
Let $A$ be a Noetherian ring with an element $y \in A$. Assume that $y$ is a nonzero divisor contained in the Jacobson radical of $A$. If $A/yA$ is an integrally closed domain, then so is $A$.
\end{theorem}

An application of Proposition \ref{local-lifting} is a new proof of Theorem \ref{NormalRing}, which was already mentioned in \cite[Remark 3.8]{EHS23}. However, it seems that most existing research articles study $(\bf{Lift})$ extensively, while $(\bf{Gr})$ is less tractable. As a side topic, we prove an ideal-theoretic analogue of Theorem \ref{NormalRing} under some conditions.

\begin{proposition}
\label{NormalIdeal}
Let $(A,\fm,k)$ be a Noetherian local ring with a nonzero divisor $y \in \fm$, and let $I \subseteq \fm$ be an ideal. Assume that $y$ is a nonzero divisor on $A/I^n$ for all $n>0$, the image of $I$ in $A/yA$ is a normal ideal, and $A/yA$ is an integrally closed domain. Then $I$ is a normal ideal in $A$.  
\end{proposition}

\begin{proof}
Let $\mathscr R_+(I)$ be the Rees ring with respect to $I$. Then using the presentation $\mathscr R_+(I)=\bigoplus_{n\geq 0} I^n$, we get
$\mathscr R_+(I)/y\mathscr R_+(I) \cong \bigoplus_{n\geq 0} I^n/yI^n$.

Fix $n>0$. Now the image of $I^n$ in $A/yA$ is equal to $I^n/(I^n \cap yA)$. We prove that $I^n/yI^n=I^n/(I^n \cap yA)$. It is clear that $yI^n \subseteq I^n \cap yA$ and so we prove that the reverse. Let $a \in I^n \cap yA$ be any element and write $a=yb$ for some $b \in A$. Since $a \in I^n$ and $y$ is a nonzero divisor of $A/I^n$ by assumption, it follows that $b \in I^n$. Thus, $a=yb \in yI^n$. We get
$$
\mathscr R_+(I)/y\mathscr R_+(I) \cong \bigoplus_{n\geq 0} I^n/(I^n \cap yA),
$$
which is the Rees ring $(A/yA)[(I/(I \cap yA))t] \subseteq (A/yA)[t]$. By assumption, $I/(I \cap yA) \subseteq A/yA$ is a normal ideal and $A/yA$ is integrally closed, so $(A/yA)[(I/(I \cap yA))t]$ is integrally closed in a normal domain $(A/yA)[t]$ in view of \cite[Proposition 5.2.1]{SwHu06}. Hence $\mathscr R_+(I)/y\mathscr R_+(I)$ is an integrally closed domain. Let $\mathfrak{M}:=\fm \oplus \bigoplus_{n\geq 1} I^nt^n$, which is the unique graded maximal ideal of $\mathscr R_+(I)$. Since $y \in \mathfrak{M}$, it follows that $\mathscr R_+(I)_{\mathfrak{M}}$ is a local normal domain. By applying \cite[Proposition 2.2]{Shi23}, we find that $\mathscr R_+(I)$ is an integrally closed domain. In particular, this ring is integrally closed in $A[t]$, which gives that $I$ is a normal ideal by \cite[Proposition 5.2.1]{SwHu06}.
\end{proof}

\begin{question}
Does Proposition \ref{NormalIdeal} hold true without assuming that $y$ is a nonzero divisor on $A/I^n$, or $A/yA$ is integrally closed?
\end{question}

It will also be interesting to investigate the connection of Proposition \ref{NormalIdeal} with the notion of \textit{superficial elements} (see \cite{Bon05} for  
a good account).

\begin{proposition}
\label{local-lifting2}
Fix a subcategory $\mathbf{D} \subseteq \mathbf{C}$. Suppose that $\mathcal{P}$ satisfies $(\bf{P1})$ and $(\bf{P3})$ in $\mathbf{D}$ and that $(\bf{Lift})$ and $(\bf{MR})$ are solved in $\mathbf{D}$. Then $(\bf{Gr})$ is solved in $\mathbf{D}$.
\end{proposition}

\begin{proof}
Let $(A,\fm)$ be a local ring in $\mathbf{D}$ and let $I \subseteq \fm$ be an ideal such that $G(I)=\bigoplus_{n \ge 0}I^n/I^{n+1}$ has $\mathcal{P}$. Then there is an isomorphism:
$$
G(I) \cong \mathscr R(I)/t^{-1}\mathscr R(I),
$$
where $\mathscr R(I)=A[It,t^{-1}]$ is the extended Rees ring. Note that $\mathscr R(I)$ has a unique graded maximal ideal $\mathfrak{M}=\fm \mathscr R(I)+It\mathscr R(I)+t^{-1}\mathscr R(I)$ containing $t^{-1}$. Moreover, $t^{-1}$ is a nonzero divisor of $\mathscr R(I)$. By $(\bf{Lift})$ and $(\bf{P1})$, we find that the localization $\mathscr R(I)_{\mathfrak{M}}$ has $\mathcal{P}$. Then $(\bf{MR})$ shows that $\mathscr R(I)$ has $\mathcal{P}$. By $(\bf{P1})$, the localization $\mathscr{R}(I)[t]=A[t,t^{-1}]$ has $\mathcal{P}$. Since $A[t,t^{-1}]$ is faithfully flat over $A$, the condition $(\bf{P3})$ allows us to conclude that $A$ has $\mathcal{P}$, as desired.
\end{proof}

We establish the Matijevic-Roberts type theorem for seminormality and weak normality.

\begin{theorem}
\label{Matijevic-Roberts1}
Let $A=\bigoplus_{n \in \mathbb{Z}}A_n$ be a graded Noetherian ring such that the integral closure of $A$ in its total ring of fractions is $\mathbb{Z}$-graded. 
\begin{enumerate}
\item
$A$ is seminormal if and only if the localization $A_{\fm}$ is seminormal for every graded maximal ideal $\fm$ of $A$.

\item
$A$ is weakly normal if and only if the localization $A_{\fm}$ is weakly normal for every graded maximal ideal $\fm$ of $A$. 
\end{enumerate}
\end{theorem}

\begin{proof}
Before starting the proof of each assertion, we prove that, under the seminormal or weak normal condition on each localized ring $A_\fm$, that $A$ is a reduced ring. Recall that seminormal rings and weakly normal rings are always reduced. Let us consider the canonical diagonal mapping
$$
f:A \to \prod_{\fm} A_\fm,
$$
where the product ranges over all graded maximal ideals. Since $A_\fm$ is reduced under the stated assumptions, it suffices to show that $f$ is injective. Assume that $x \in A$ is a nonzero element in the kernel of $f$. Let $N:=Ax \subseteq A$ and consider an associated prime $\fp$ of $N$. It is known that every associated prime ideal of a graded ring is graded by \cite[Lemma 1.5.6]{BrHer93}, which shows that $\fp \subseteq \fm$ for some graded maximal ideal $\fm \subseteq A$, which we fix now. Then $0 \ne (Ax)_\fp \subseteq A_\fp$. Since $f(Ax)=0$, it follows that $(Ax)_\fm=0$. But as we have a factorization $Ax \subseteq A \to A_\fm \to A_\fp$, it follows that $(Ax)_\fp=0$, which gives a contradiction. So we proved that $f$ is injective and thus, $A$ is reduced. In particular, the total ring of fractions of $A$ is a finite product of fields. So the integral closure of $A$ in the total ring of fractions of $A$ is a (possibly non-Noetherian) normal $\mathbb{Z}$-graded ring by hypothesis. We denote this integral closure by $\overline{A}$.

$(1)$: Since seminormality is a local property in view of \cite[Proposition 3.7]{S80}, the ``only if" part is clear. It suffices to prove the ``if" part. So assume that $A_\fm$ is seminormal for every graded maximal ideal $\fm$. Again by \cite[Proposition 3.7]{S80}, it suffices to show that $A_\fp$ is seminormal for every maximal ideal $\fp \subseteq A$. Let $\fp^*$ be the ideal generated by all homogeneous elements contained in $\fp$. Then it follows from \cite[Lemma 1.5.6]{BrHer93} that $\fp^*$ is a prime ideal. Let $A_{(\fp)}$ and $A_{(\fp^*)}$ be the homogeneous localizations as in \cite[p.31]{BrHer93}. Then by the construction, we have $A_{(\fp^*)}=A_{(\fp)}$. Since $\fp^*$ is graded, there is a graded maximal ideal $\fm$ such that $\fp^* \subseteq \fm$. Then we have a localization map $A_{(\fm)} \to A_{(\fp^*)}=A_{(\fp)}$, which extends to
$$
A_{(\fm)} \to A_{(\fp^*)}=A_{(\fp)} \to A_\fp,
$$
where the second map is also a localization. So to prove that $A_\fp$ is seminormal, it suffices to prove that the same property holds on $A_{(\fm)}$. Without loss of generality, we may replace $A$ by $A_{(\fm)}$ (resp. $\overline{A}$ by $\overline{A}_{(\fm)}:=\overline{A} \otimes_A A_{(\fm)}$) to assume that $A$ is a $\mathbb{Z}$-graded Noetherian ring with a unique graded maximal ideal $\fm$ such that $A_\fm$ is seminormal, and $\overline{A}$, which is the integral closure of $A$, is a normal $\mathbb{Z}$-graded ring. We have a push-out diagram:
$$
\begin{CD}
A @>>> \overline{A} \\
@VVV @VVV \\
A_\fm @>>> \overline{A}_\fm=\overline{A} \otimes_A A_\fm \\
\end{CD}
$$
Suppose that $x \in \overline{A}$ is a homogeneous element that satisfies $x^2,x^3 \in A$. Since $\overline{A}_\fm$ coincides with the integral closure of $A_\fm$ in its total ring of fractions, it follows from the hypothesis that $\frac{x}{1} \in A_\fm$. Since $x$ is a homogeneous element, we can write
$$
\frac{x}{1}=\frac{a}{b}~\mbox{for homogeneous elements}~a \in A~\mbox{and}~b \in A \setminus \fm.
$$
However, as the only graded maximal ideal of $A$ is $\fm$ and $b$ is a homogeneous element not contained in $\fm$, we find that $b$ is a unit element of $A$. That is, we get $x \in A$. By applying \cite[Proposition 2.4]{LV85}  (see also \cite[Theorem 2]{And81}), $A$ is seminormal, as desired.

$(2)$: The weak normality is a local property by \cite[Corollary 2]{Y85}. So as in the first case, we may assume that $A$ is a $\mathbb{Z}$-graded ring with a unique maximal ideal $\fm$ and $A_\fm$ is weakly normal. Keep the notation as in $(1)$. In view of \cite[Proposition 3.1]{VL99}, it suffices to show that if a homogeneous element $x \in \overline{A}$ satisfies $x^p \in A$ and $px \in A$ for some prime $p$, then $x \in A$. By the weak normality of $A_\fm$, it follows that $\frac{x}{1} \in A_\fm$ and one can proceed as in the first case to finish the proof.
\end{proof}

We recall the following result.

\begin{theorem}[Heitmann, Murayama]
\label{LifNormality}
Let $(A,\fm)$ be a Noetherian local ring with a nonzero divisor $y \in \fm$. If $A/yA$ is seminormal (resp. weakly normal), then $A$ is also seminormal (resp. weakly normal).
\end{theorem}

While the seminormal case due to Heitmann is found in \cite{He08}, the weakly normal case due to Murayama is found in \cite[Proposition 4.10]{Mu22}, where a new proof of the seminormal case is also given. Recently, even another innovative approach is discovered in \cite{BEGMNW22}. The following result gives a realization of Proposition \ref{local-lifting2}. The corresponding statement in the normal case has been already known. See for instance \cite[Theorem 4.5.9]{BrHer93}.

\begin{theorem}
\label{grlocal1}
Let $(A,\fm)$ be a Noetherian local ring of $\dim A \geq 1$ and let $I$ be an ideal of $A$. Then the following assertions hold.
\begin{enumerate}
\item
If $G(I)$ is seminormal, then $\mathscr R_+(I)$, $\mathscr R(I)$ and $A$ are seminormal.

\item
If $G(I)$ is weakly normal, then $\mathscr R_+(I)$, $\mathscr R(I)$ and $A$ are weakly normal.
\end{enumerate}
\end{theorem}

\begin{proof}
$(1)$: First, we prove that the extended Rees ring $\mathscr R(I)$ is seminormal. Recall that $G(I) \cong \mathscr R(I)/t^{-1} \mathscr R(I)$ and $t^{-1}$ is a nonzero divisor of $\mathscr R(I)$. Since $A$ is local, $\mathscr R(I)$ has a unique graded maximal ideal $\fM:=\fm \mathscr R(I)+It\mathscr R(I)+t^{-1}\mathscr R(I)$ which contains $t^{-1}$. By Theorem \ref{LifNormality}, the localization $\mathscr R(I)_{\fM}$ is seminormal. Since $G(I)$ is reduced, it follows that the ring $\mathscr R(I)$ is also reduced by \cite[Exercise 5.9 at page 116]{SwHu06}. Moreover, the integral closure of $\mathscr R(I)$ in its total ring of fractions is a $\mathbb{Z}$-graded ring in view of \cite[Proposition 5.2.4]{SwHu06}. So Theorem \ref{Matijevic-Roberts1} yields seminormality of $\mathscr R(I)$. Then the localization $\mathscr R(I)[t]=A[t,t^{-1}]$ is seminormal. Since $A \to A[t,t^{-1}]$ is faithfully flat, it follows that $A$ is seminormal by \cite[Theorem 2.22]{V11}. The reducedness of $G(I)$ gives that $I$ is a normal ideal, that is $I^n$ is integrally closed for all $n>0$ in view of \cite[Exercise 5.7 at page 116]{SwHu06}. So $\mathscr R_+(I)$ is integrally closed in the seminormal ring $A[t]$ by \cite[Proposition 5.2.1]{SwHu06}. Hence $\mathscr R_+(I)$ is seminormal as well.

$(2)$: One can proceed as in the first case by applying Theorem \ref{Matijevic-Roberts1}. So it suffices to recall that $(\bf{Lift})$ holds for $\mathcal{P}=\rm{weak~normality}$ by Theorem \ref{LifNormality}.
\end{proof}

\begin{remark}
We applied Proposition \ref{local-lifting2} to prove Theorem \ref{grlocal1}. However, we do not have to worry about specifying what the subcategory $\mathbf{D} \subseteq \mathbf{C}$ is, because the ideas appearing in the proof of Proposition \ref{local-lifting2} are mostly essential.
\end{remark}

\begin{acknowledgement}
Both authors would like to thank the anonymous referee for pointing out errors in the first draft.
\end{acknowledgement}

\end{document}